\documentclass [twoside,reqno, draft,  12pt] {amsart}

\usepackage{amsfonts}
\usepackage{amssymb}
\usepackage{a4}
\usepackage{color}

\newtheorem{thm}{Theorem}[section]
\newtheorem{cor}[thm]{Corollary}
\newtheorem{lem}[thm]{Lemma}
\newtheorem{prop}[thm]{Proposition}

\theoremstyle{definition}

\numberwithin{equation}{section}

\renewcommand{\Re}{\hbox{Re}\,}
\renewcommand{\Im}{\hbox{Im}\,}

\newcommand{\C}{\mathbb{C}}

\newcommand{\R}{\mathbb{R}}

\newcommand{\supp}{\operatorname{supp}}

\def\tilde{\widetilde}
\def \bfo {\begin {eqnarray*} }
\def \efo {\end {eqnarray*} }
\def \ba {\begin {eqnarray*} }
\def \ea {\end {eqnarray*} }
\def \beq {\begin {eqnarray}}
\def \eeq {\end {eqnarray}}
\def \supp {\hbox{supp }}

\def \p {\partial}

\def\tilde{\widetilde}
\def \bfo {\begin {eqnarray*} }
\def \efo {\end {eqnarray*} }
\def \ba {\begin {eqnarray*} }
\def \ea {\end {eqnarray*} }
\def \beq {\begin {eqnarray}}
\def \eeq {\end {eqnarray}}
\def \supp {\hbox{supp }}

\def \p {\partial}

\begin{document}

\title[Global identifiability in acoustic tomography of moving fluid]{Global identifiability of low regularity fluid parameters  in acoustic tomography of moving fluid}

\author[Liu]{Boya Liu}
\address{B. Liu, Department of Mathematics\\
University of California, Irvine\\ 
CA 92697-3875, USA }
\email{boyal2@uci.edu}

\maketitle
	
\begin{abstract} 
We are concerned with inverse boundary problems for first order perturbations of the Laplacian, which arise as model operators in the acoustic tomography of a moving fluid. We show that the knowledge of the Dirichlet--to--Neumann map on the boundary of a bounded domain in $\R^n$,  $n\ge 3$, determines the first order perturbation of low regularity up to a natural gauge transformation, which sometimes is trivial.  As an application, we recover the fluid parameters of low regularity from boundary measurements, sharpening 
the regularity assumptions in the recent results of \cite{Agal2015} and \cite{AARN2015}. In particular, we allow some fluid parameters to be discontinuous.

\end{abstract}
	
\section{Introduction and statement of results}
	
Let $\Omega\subset \R^n$, $n\ge 3$, be a bounded open set with smooth boundary and let us consider a moving fluid in $\Omega$ characterized by the sound speed $c$, density $\rho$, fluid velocity vector $v$, and absorption coefficient $\alpha = \alpha(\cdot, \omega)$ at a fixed frequency $\omega>0$. The time-harmonic acoustic pressure $p(x,t)$ is of the form  $p(x,t)= \Re (u(x)e^{-i\omega t})$, where $u$ satisfies the following second order elliptic equation, 
\begin{equation}
\label{eq_1_1}
L_{A(\omega),q(\omega)}u:= (-\Delta - 2iA(\omega)\cdot \nabla  + q(\omega))u=0\quad \textrm{in}\quad \Omega, 
\end{equation}
 with 
 \begin{equation}
 \label{eq_int_acoustic_tom_coeff}
 \begin{aligned}
&A(x;\omega) = \frac{\omega v(x)}{c^2(x)} + \frac{i}{2}\frac{\nabla \rho(x)}{\rho(x)}, \\ 
&q(x;\omega) = -\frac{\omega ^2}{c^2(x)} - 2i\omega \frac{\alpha (x;\omega)}{c(x)}, \quad \alpha(x,\omega)=\omega^{\zeta (x)}\alpha_0(x).
\end{aligned} 
\end{equation}
This model was considered for instance in \cite{Agal2015},  \cite{AN_scatt}, \cite{AARN2015},   \cite{BBS}, \cite{BSZR}, \cite{DRKW},  \cite{RBKS}, \cite{RumSh}, and \cite{MRHE}. In this paper we are interested in the inverse boundary problem of identifiability of the fluid parameters $c$, $\rho$, $v$, and  $\alpha$ from boundary measurements.  Such an inverse problem has applications in ocean tomography, where one wishes to determine the ocean temperature and heat transferring currents from acoustic measurements, and in medical diagnostics, where scalar inhomogeneities and the blood flow are to be determined.

To state the problem in precise terms, let us assume that $A(\omega)\in L^\infty (\Omega;\C^n)$ and $q(\omega)\in L^\infty(\Omega;\C)$, and let us first observe that the operator 
\[
L_{A(\omega),q(\omega)}:H^1_0(\Omega)\to H^{-1}(\Omega)
\]
is Fredholm of index zero. Here and in what follows the spaces $H^s(\Omega)$, $H^s_0(\Omega)$, $s\in \R$, are the standard $L^2$--based Sobolev spaces on $\Omega$, see \cite{Eskin_book}.  Indeed, the operator $L_{A(\omega),q(\omega)}$ differs from the invertible operator $-\Delta: H^1_0(\Omega)\to H^{-1}(\Omega)$ by a compact perturbation. We shall make the following standing assumption: 

\textbf{Assumption A.} The operator $L_{A(\omega),q(\omega)}:H^1_0(\Omega)\to H^{-1}(\Omega)$ is injective. 

Hence, it follows that for any $f\in H^{1/2}(\p \Omega)$, the boundary value problem 
\begin{equation}
\label{eq_int_BVP}
\begin{aligned}
&L_{A(\omega),q(\omega)}u=0\quad \text{in}\quad \Omega,\\
&u|_{\p\Omega}=f,
\end{aligned}
\end{equation}
has a unique solution $u\in H^1(\Omega)$, depending continuously on $f$.  Let $\nu$ be the unit outer normal vector to the boundary of $\Omega$. We shall define the trace of  the normal derivative $\partial_{\nu} u \in H^{-1/2} (\partial \Omega)$ as follows. Let $\varphi \in H^{1/2}(\partial \Omega)$ and let $v \in H^1(\Omega)$ be a continuous extension of $\varphi$. We define 
\begin{equation}
\label{tracedefinition}
\langle \partial _\nu u, \varphi \rangle_{H^{-1/2}(\partial \Omega) \times H^{1/2}(\partial \Omega)} = \int_\Omega [\nabla u \cdot \nabla v - 2i(A(\omega) \cdot \nabla u)v + q(\omega)uv]dx.
\end{equation}
Here $\langle \cdot, \cdot \rangle$ is the distributional duality. As $u$ solves \eqref{eq_int_BVP}, the definition of the trace $\partial  _\nu u$ on $\partial \Omega$ is independent of the choice of an extension $v$ of $\varphi$. 
The Dirichlet--to--Neumann map $\Lambda_{A(\omega),q(\omega)}$ is then defined as follows,
\[
\Lambda_{A(\omega),q(\omega)}: H^{1/2}(\p \Omega)\to H^{-1/2}(\p \Omega),\quad f\mapsto \p_\nu u|_{\p \Omega}. 
\]

The problem of acoustic tomography of moving fluids that we are interested in is as follows: given the Dirichlet--to--Neumann map $\Lambda_{A(\omega), q(\omega)}$ for some frequencies $\omega>0$, determine the fluid parameters $c$, $\rho$, $v$ and $\alpha$ in $\Omega$, see \eqref{eq_int_acoustic_tom_coeff}. This problem was studied in  \cite{Agal2015},  \cite{AARN2015}, under the assumption that $\Omega$ is simply connected and that the fluid parameters enjoy the following regularity properties, 
\begin{equation}
\label{eq_int_regul_coeff}
\begin{aligned}
&c\in W^{1,\infty}(\Omega;\R), \quad c>0, \quad \rho\in C(\overline{\Omega})\cap C^2(\Omega),\quad \rho>0, \\
&v\in W^{1,\infty}(\Omega;\R^n),\quad \alpha_0\in C(\overline{\Omega};\R), \quad \zeta\in C(\overline{\Omega};\R),\quad \zeta\ne 0.
\end{aligned}
\end{equation}
Specifically, it was shown in \cite{Agal2015} that the knowledge of the Dirichlet--to--Neumann map $\Lambda_{A(\omega), q(\omega)}$ at a fixed frequency $\omega>0$ determines the sound speed $c$ and the fluid velocity $v$ in $\Omega$ uniquely provided that $\alpha=0$ and $\rho$ is a constant.  In  \cite{AARN2015} it is proven that the knowledge of the Dirichlet--to--Neumann map $\Lambda_{A(\omega), q(\omega)}$ at two distinct frequencies determines $c$, $v$ in $\Omega$ uniquely, and $\rho$ up to a multiplication by a constant provided that $\alpha=0$, and  the knowledge of the Dirichlet--to--Neumann map $\Lambda_{A(\omega), q(\omega)}$ at three distinct frequencies determines $c$, $v$, $\alpha$ in $\Omega$ uniquely, and $\rho$ up to a multiplication by a constant.  The works  \cite{Agal2015}, \cite{AARN2015} discuss also the two-dimensional case, see also \cite{AN_scatt} for the study of  the corresponding inverse scattering problem in dimension $2$.  

The crucial idea of \cite{Agal2015}, \cite{AARN2015}  is that thanks to the regularity assumptions \eqref{eq_int_regul_coeff}, one has $A(\omega)\in W^{1,\infty}(\Omega;\C^n)$, $q(\omega)\in L^\infty(\Omega;\C)$, and one can then view the operator $L_{A(\omega),q(\omega)}$ as   the magnetic Schr\"odinger operator 
\begin{equation}
\label{magneticS}
(D_x+A(\omega))^2 + \tilde{q}(\omega), \quad D_x = \frac{1}{i}\partial_{x}, 
\end{equation}
where  
\begin{equation}
\label{qA1}
\quad \tilde{q}(\omega) = q(\omega) + i (\nabla \cdot A(\omega)) -A(\omega)^2 \in L^\infty(\Omega; \mathbb{C}). 
\end{equation}
These regularity assumptions allow one to use the global uniqueness result for the inverse boundary problem for the magnetic Schr\"odinger operator with electromagnetic potentials of class $L^\infty$, established in \cite{KKGU2014}. 

The purpose of this paper is to weaken the regularity assumptions \eqref{eq_int_regul_coeff} on the fluid parameters $c$, $v$, $\rho$, and $\alpha$ in the results of   \cite{Agal2015},  \cite{AARN2015}, and in particular, to allow some parameters to be discontinuous. In doing so, motivated by the recent work \cite{KKGU2017geo}, we shall impose regularity assumptions on the fluid parameters, weaker than those in \eqref{eq_int_regul_coeff}, implying that $A(\omega)\in (H^1 \cap L^\infty)(\Omega;\C^n)$ and $q(\omega)\in L^\infty(\Omega;\C)$.  In this case the reduction to the magnetic Schr\"odinger operator \eqref{magneticS} is not useful since $\tilde{q}(\omega)$ given by \eqref{qA1} is no longer in $L^\infty(\Omega)$ and we only have $\tilde q(\omega)\in L^2(\Omega)$. To the best of our knowledge, there are no results available for the inverse boundary problem for the magnetic Schr\"odinger operator with electromagnetic potentials of such low regularity. See \cite{Haberman_magnetic}, \cite{KKGU2014}  for the sharpest results in this direction. Therefore, we shall deal with our inverse problem directly, relying on the techniques developed in \cite{KKGU2017geo}.

Our results for the problem of acoustic tomography of moving fluids will be a consequence of a general result for the following inverse boundary problem, which we shall study first. To state the problem, let 
\[
L_{A,q}= -\Delta - 2iA\cdot \nabla  + q,
\]
where $A\in (H^1 \cap L^\infty)(\Omega;\C^n)$ and $q\in L^\infty(\Omega;\C)$, and let us suppose that assumption (A) holds for $L_{A,q}$. Notice that here $A$ and $q$ need not be of the form \eqref{eq_int_acoustic_tom_coeff}.  The inverse problem under consideration is to determine $A$ and $q$ in $\Omega$, given the Dirichlet--to--Neumann map $\Lambda_{A,q}$. Similarly to the inverse boundary problem for the magnetic Schr\"odinger operator,  there is an obstruction to uniqueness in this problem given by the following gauge transformation: if $\varphi \in C^\infty (\overline{\Omega})$ and $\varphi|_{\partial \Omega} = 0, \partial_\nu \varphi|_{\partial \Omega} = 0$, then $\Lambda_{A_1, q_1} = \Lambda_{A_2, q_2}$ where  
\begin{equation}
\label{gaugetransformation}
A_1 = A_2 + \nabla \varphi \quad \textrm{and} \quad q_1 = q_2 + 2A_2 \cdot \nabla \varphi +(\nabla \varphi)^2 -i\Delta \varphi,
\end{equation} 
see \cite{Sun_1993}. This follows from the fact that  
\begin{equation}
\label{eq_conj_lem}
e^{-i\varphi} L_{A_2,q_2} (e^{i\varphi}u)=L_{A_1, q_1}u
\end{equation}
and 
\begin{equation}
\label{eq_conj_lem_2}
u|_{\partial \Omega} = (e^{i\varphi }u)|_{\partial \Omega}, \quad (\partial_\nu u)|_{\partial \Omega} = \partial _\nu (e^{i\varphi}u)|_{\partial \Omega}
\end{equation}
Therefore, one can only hope to recover the coefficients $A$ and $q$ of the operator $L_{A,q}$ in $\Omega$ up to the gauge transformation \eqref{gaugetransformation} from the knowledge of the Dirichlet--to--Neumann map $\Lambda_{A,q}$. In this direction,  our main result is as follows.
\begin{thm}
\label{thm_main_complex}
Let $\Omega\subset \R^n$, $n\ge 3$, be a bounded open set with smooth boundary. Suppose that $A_1,A_2\in (H^1 \cap L^\infty)(\Omega;\C^n)$ and $q_1,q_2\in L^\infty(\Omega;\C)$ are such that assumption \emph{(A)} holds for $L_{A_1,q_1}$ and $L_{A_2,q_2}$. If $\Lambda_{A_1,q_1}=\Lambda_{A_2, q_2}$, then there exists $\varphi \in W^{1, \infty}(\Omega; \mathbb{C})$ such that $A_1 = A_2 + \nabla \varphi$ and $q_1 = q_2+2A_2 \cdot \nabla \varphi +(\nabla \varphi)^2 -i\Delta \varphi$ in $\Omega$. 
\end{thm}
	
Our first corollary is a generalization of a result of \cite{Agal2015} to the case of the vector field $A$ of low regularity.	 Notice that here the vector fields and scalar potentials are both real-valued.
\begin{cor}
\label{thm_main_real}
Let  $\Omega\subset \R^n$, $n\ge 3$, be a bounded open set with smooth boundary. Suppose that $A_1,A_2\in (H^1 \cap L^\infty)(\Omega;\R^n)$ and $q_1,q_2\in L^\infty(\Omega;\R)$ are such that assumption \emph{(A)} holds for $L_{A_1,q_1}$ and $L_{A_2,q_2}$. If $\Lambda_{A_1,q_1}=\Lambda_{A_2, q_2}$, then $A_1=A_2$ and $q_1=q_2$ in $\Omega$. 
\end{cor}

\textbf{Remark.} The assumption $A \in (H^1 \cap L^\infty)(\Omega;\C^n)$ in Theorem \ref{thm_main_complex} corresponds to the optimal space on the scale of spaces $H^s \cap L^\infty$, $s\ge 0$, for which the inverse boundary problem for the operator $L_{A,q}$ can be solved by means of the techniques of $L^2$ Carleman estimates.This could be seen in particular from the estimate \eqref{eq_500_3} which is of purely qualitative nature, see also the discussion in \cite{KKGU2017geo}.

Let us now return to the problem of acoustic tomography of moving fluid. Let $L_{A_j(\omega),q_j(\omega)}$ be of the form \eqref{eq_1_1}, \eqref{eq_int_acoustic_tom_coeff}, $j=1,2$.  We make the following regularity assumptions on the fluid parameters, which are strictly weaker than those in \eqref{eq_int_regul_coeff}, 
\begin{equation}
\label{eq_coeff_our}
\begin{aligned}
&c_j\in (H^1\cap L^\infty) (\Omega;\R), \quad c_j\ge c_0>0,  \quad \rho\in (H^2\cap W^{1,\infty})(\Omega;\R),\quad \rho_j\ge \rho_0>0, \\
&v_j\in (H^1\cap L^\infty) (\Omega;\R^n),\quad \alpha_{0,j}\in L^\infty(\Omega;\R), \quad \zeta_j\in L^\infty(\Omega;\R), \quad \zeta_j\ne 0.
\end{aligned}
\end{equation} 
In particular, we allow the parameters $c_j$, $v_j$, and  $\alpha_j$ to be discontinuous.   The regularity assumptions \eqref{eq_coeff_our} imply that $A_j(\omega)\in (H^1 \cap L^\infty)(\Omega;\C^n)$ and $q_j(\omega)\in L^\infty(\Omega; \C)$, $j=1,2$.

The following direct consequence of Corollary \ref{thm_main_real} gives an improvement of the corresponding result of \cite{Agal2015} in terms of the regularity of the parameters $c_j, v_j$, and can be stated as follows. 
\begin{cor}
\label{onefrequency}
Let $\Omega\subset \R^n$, $n\ge 3$, be a bounded open set with smooth boundary. Assume that $\alpha_1=\alpha_2=0$, $\rho_1, \rho_2$ are constant in $\Omega$, and that  $c_1$, $c_2$, $v_1$, $v_2$ satisfy \eqref{eq_coeff_our}.  Suppose that assumption \emph{(A)} holds for $L_{A_1(\omega),q_1(\omega)}$ and $L_{A_2(\omega),q_2(\omega)}$, and that $\Lambda_{A_1(\omega),q_1(\omega)}=\Lambda_{A_2(\omega),q_2(\omega)}$ for a fixed frequency $\omega>0$. Then $c_1=c_2$ and $v_1=v_2$ in $\Omega$.
\end{cor}

The following results provide improvements of the corresponding results of \cite{AARN2015}  in terms of the regularity of the fluid parameters, and can be stated as follows. 
\begin{cor}
\label{twofrequency}
Let $\Omega\subset \R^n$, $n\ge 3$, be a bounded open set with smooth connected boundary. Assume that $\alpha_1=\alpha_2=0$, 
and that $c_1$, $c_2$, $v_1$, $v_2$, $\rho_1$, $\rho_2$ satisfy \eqref{eq_coeff_our}. 
Suppose that assumption \emph{(A)} holds for $L_{A_1(\omega),q_1(\omega)}$ and $L_{A_2(\omega),q_2(\omega)}$, and $\Lambda_{A_1(\omega),q_1(\omega)}=\Lambda_{A_2(\omega),q_2(\omega)}$ for two fixed distinct frequencies $\omega=\omega_1>0$ and $\omega=\omega_2>0$. Then $c_1=c_2$, $v_1=v_2$, and $\rho_1=C\rho_2$ in $\Omega$, where $C>0$ is a constant.
\end{cor}

\begin{cor}
\label{threefrequency}
Let $\Omega\subset \R^n$, $n\ge 3$, be a bounded open set with smooth connected boundary. Assume that   $c_1$, $c_2$, $v_1$, $v_2$, $\rho_1$, $\rho_2$, $\alpha_1$, $ \alpha_2$  satisfy \eqref{eq_coeff_our}. 
  Suppose that assumption \emph{(A)} holds for $L_{A_1(\omega),q_1(\omega)}$ and $L_{A_2(\omega),q_2(\omega)}$, and $\Lambda_{A_1(\omega),q_1(\omega)}=\Lambda_{A_2(\omega),q_2(\omega)}$ for three fixed distinct frequencies $\omega=\omega_1>0$, $\omega=\omega_2>0$, $\omega=\omega_3>0$. Then $c_1=c_2$, $v_1=v_2$, $\rho_1=C\rho_2$, and $\alpha_1=\alpha_2$ in $\Omega$.
\end{cor}

Let us emphasize that the issues of relaxing regularity assumptions on the coefficients in inverse boundary problems have recently been studied extensively. Let us illustrate this by discussing the following three fundamental inverse problems: the Calder\'on problem, 
and inverse boundary problems for the magnetic Schr\"odinger operator, as well as for the advection diffusion equation. The work  \cite{Syl_Uhl_1987} established the global identifiability for $C^2$ conductivities in the Calder\'on problem, see also \cite{Novikov_1988} for the global uniqueness in the closely related inverse scattering problem at a fixed energy. Subsequently, the regularity of the conductivity has been relaxed in \cite{Brown_1996}, \cite{Brown_Torres_2003}, \cite{Paiv_Pan_Uhl}  to conductivities having $3/2$ derivatives in a suitable sense. The paper \cite{Hab_Tataru} obtained the global uniqueness for $C^1$ conductivities and Lipschitz continuous conductivities close to the identity.  The latter smallness condition was removed in \cite{Caro_Rogers}, and the global uniqueness for conductivities in $W^{1,n}$, with $n=3,4$, was obtained in \cite{Haberman}.   

Turning the attention to the inverse problem for the magnetic Schr\"odinger operator, starting with the work \cite{Sun_1993}, where the global identifiability of the magnetic field and electric potential was established for magnetic potentials in $W^{2,\infty}$, satisfying a smallness condition, there has been a substantial amount of work reducing the regularity of the magnetic potential, see  \cite{Haberman_magnetic}, \cite{KKGU2014}, \cite{NakSunUlm_1995}, \cite{Panchenko_2002}, \cite{Salo_diss},   \cite{Tolmasky_1998}. The sharpest results in terms of the regularity of the magnetic potentials are given in \cite{KKGU2014} and \cite{Haberman_magnetic}, where the global identifiability is obtained for $L^\infty$ magnetic and electric potentials in any dimension $n\ge 3$ and for small magnetic potentials in $W^{s, 3}$, $s>0$,  and electric potentials in $W^{-1,3}$, in dimension $n=3$. 

We shall finally make some comments concerning the inverse boundary problem for the advection diffusion equation. Starting with
\cite{Cheng_Nakamura_Somersalo}, regularity issues for this problem were addressed in \cite{Knudsen_Salo}, \cite{KKGU2017geo}, \cite{Pohjola}, and \cite{Salo_diss}. The sharpest results in terms of the regularity of the advection term are due to \cite{Pohjola} and \cite{KKGU2017geo}, showing the global uniqueness in the inverse boundary problem for the advection diffusion equation, with a H\"older continuous advection term, with the H\"older exponent in the range $(2/3,1]$, and with the advection term of class $H^1\cap L^\infty$, respectively. 
 
The paper is organized as follows. Section \ref{sec_CGO} is devoted to the construction of complex geometric optics solutions for our equations, relying on the techniques developed in \cite{KSU_2007}, \cite{KKGU2014}, \cite{KKGU2017geo}, and \cite{Salo_Tzou_2009}. Theorem \ref{thm_main_complex} together with Corollary \ref{thm_main_real} are established in Section \ref{sec_proof_main}, and Section \ref{tomography} is concerned with the proofs of Corollary \ref{twofrequency} and Corollary \ref{threefrequency}, following the approach of \cite{AARN2015}  and noticing that it still works in our low regularity setting. Appendix \ref{boundary} reviews the boundary determination of the vector field $A$ of class $H^1\cap L^\infty$ from boundary measurements.

\section{Construction of complex geometric optics solutions} 
	
\label{sec_CGO}

We begin this section by introducing the following  operator, which comprises \eqref{eq_1_1} and its formal $L^2$ adjoint,
\begin{equation}
\label{pvwq}
P_{V, W, q} = -\Delta + V\cdot \nabla  +(\nabla \cdot W) +q,
\end{equation} where $V,W \in  L^{\infty} (\Omega; \mathbb{C}^n)$ and $q \in L^{\infty}(\Omega; \mathbb{C})$. Here the divergence $\nabla \cdot W$ is defined in the sense of distribution theory and we have $(\nabla \cdot W)u = \nabla \cdot (uW) - \nabla u \cdot W \in H^{-1}(\Omega)$ for $u \in H^1(\Omega)$. Thus, $P_{V,W,q}: H^1(\Omega) \rightarrow H^{-1} (\Omega)$. 
	
Let $0<h\le 1$ be a semiclassical parameter. Our starting point is the following solvability result for the operator $h^2P_{V,W,q}$, conjugated by an exponential weight. 
\begin{prop}
\label{solvability result}
Let $\varphi(x) = \alpha \cdot x$, where $\alpha \in \mathbb{R}^n$ is such that $|\alpha| = 1$, and assume that $V,W\in L^\infty (\Omega;\C^n)$ and $q\in L^\infty(\Omega;\C)$. If $h>0$ is small enough, then for any $v \in H^{-1}(\Omega)$, there is a solution $u \in H^1(\Omega)$ of the equation 
\begin{equation}
e^{\frac{\varphi}{h}}(h^2P_{V,W,q})e^{-\frac{\varphi}{h}}u = v  \quad \textrm{in} \quad \Omega
\end{equation}
satisfying $\|u\|_{H^1_{\textrm{scl}}(\Omega)} \le \frac{C}{h}\|v\|_{H^{-1}_{\textrm{scl}}(\Omega)}$. 
\end{prop}
Here 
\[
\|u\|_{H^1_{\textrm{scl}}(\Omega)}^2=\|u\|_{L^2(\Omega)}^2+\|h\nabla u\|_{L^2(\Omega)}^2,\quad 
\|u\|_{H^{-1}_{\textrm{scl}}(\Omega)}=\sup_{0\ne\psi\in C^\infty_0(\Omega)} \frac{|\langle u,\psi \rangle_{\Omega}|}{\|\psi\|_{H^1_{\textrm{scl}}(\Omega)}},
\]
where $\langle \cdot,\cdot \rangle_{\Omega}$ is the distributional duality on $\Omega$.  Proposition \ref{solvability result} is established in the work \cite[Proposition 2.3]{KKGU2017geo} in the setting of admissible compact Riemannian manifolds with boundary, relying crucially on the works \cite{KSU_2007}, \cite{KKGU2014},  \cite{KKGU2017}, and  \cite{Salo_Tzou_2009}.  The arguments in the Euclidean case are exactly the same.

The goal of this section is to review the construction of complex geometric optics (CGO) solutions  for the equation  $P_{V,W,q}u=0$ in $\Omega$ with $V, W\in (H^1\cap L^\infty)(\Omega; \C^n)$ and $q\in L^\infty(\Omega;\C)$. In doing so, we shall follow the arguments of \cite{KKGU2014} and \cite{KKGU2017geo}. 
In general, CGO solutions are of the form 
\[
u(x,\zeta;h) = e^{\frac{x\cdot \zeta}{h}} (a(x,\zeta;h)+r(x,\zeta;h)),
\]
 where $\zeta \in \mathbb{C}^n, \zeta \cdot \zeta = 0, |\zeta|\sim 1,a$ is a smooth amplitude, and $r$ is a remainder term.  
 	
First, we shall extend $V$ and $W$ to compactly supported functions in $(H^1 \cap L^\infty)(\mathbb{R}^n; \mathbb{C}^n)$ and denote the extensions by the same letters. We refer to \cite[Section 2.2]{KKGU2016} for a construction of such extensions. In order to obtain nice remainder estimates for our CGO solutions, we shall work with  regularizations of $V$ and $W$. To this end, let $\psi _\tau (x) = \tau ^{-n}\psi\bigg(\dfrac{x}{\tau}\bigg), \tau>0$, be the usual mollifier with $\psi \in C_0^\infty(\mathbb{R}^n), 0\le \psi \le 1$, and $\int_{\mathbb{R}^n} \psi dx =1$. Assume also that $\psi$  is radial.  Define $V_\tau = V\ast \psi_\tau \in C_0^\infty (\mathbb{R}^n; \mathbb{C}^n)$.  We have the following estimates, which allow us to approximate $V$ by its regularization $V_\tau$, see \cite[Appendix B]{KKGU2017geo} for the proof.
\begin{prop}
\label{approximationestimates}
We have
\begin{align*}
&\|V-V_\tau\|_{L^2(\mathbb{R}^n)} = o(\tau),  \quad \|V_\tau\|_{L^2(\mathbb{R}^n)} = \mathcal{O}(1), \\ 
&\|\nabla  V_\tau\|_{L^2(\mathbb{R}^n)} = \mathcal{O}(1), \quad \|\p^{\alpha} V_\tau\|_{L^2(\mathbb{R}^n)} = o(\tau ^{-1}), \quad |\alpha|=2, \\  
&\|V_\tau\|_{L^\infty(\mathbb{R}^n)} = \mathcal{O}(1), \quad \|\nabla V_\tau\|_{L^\infty(\mathbb{R}^n)} = \mathcal{O}(\tau ^{-1}), \quad  \tau \rightarrow 0.
\end{align*}
\end{prop}
	
We want to find an amplitude $a$ and a remainder $r$ such that 
\begin{equation}
\label{3.1}
P_{V,W,q}e^{\frac{x\cdot \zeta}{h}} (a+r) = 0.
\end{equation} 
Here  we shall allow $\zeta$ to  depend slightly on $h$, i.e. $\zeta = \zeta_0 + \zeta_1$, where $\zeta_0$ is independent of $h$ and $\zeta_1 = \mathcal{O}(h)$ as $h \rightarrow 0$. We also assume $|\Re\zeta_0|$ = $|\Im\zeta_0| = 1$.
	
Writing 
\[
e^{-{\frac{x\cdot \zeta}{h}}}(h^2P_{V,W,q})e^{\frac{x\cdot \zeta}{h}}r = -e^{-{\frac{x\cdot \zeta}{h}}}(h^2P_{V,W,q})e^{\frac{x\cdot \zeta}{h}}a,
\]
and computing the conjugated operator, we get 
\begin{align*}
e^{-{\frac{x\cdot \zeta}{h}}}(h^2P_{V,W,q})e^{\frac{x\cdot \zeta}{h}} r
= &-\big(-h^2\Delta - 2\zeta_0 \cdot h\nabla -2\zeta_1 \cdot h\nabla + hV \cdot h\nabla \\ & +h\zeta_0 \cdot (V-V_\tau) + h\zeta_0 \cdot V_{\tau}	+ hV \cdot \zeta_1+ h^2(\nabla \cdot W)+h^2q\big) a.
\end{align*}

Following the WKB method, we obtain the following equation for the remainder
\begin{equation}
\label{CGO2}
\begin{split}
e^{-{\frac{x\cdot \zeta}{h}}}(h^2P_{V,W,q})e^{\frac{x\cdot \zeta}{h}}r = &-[-h^2\Delta a + hV \cdot h\nabla a + h^2(\nabla \cdot W)a + h^2qa]+ 2\zeta_1\cdot h\nabla a \\
&- h\zeta_0 \cdot (V-V_\tau)a - hV\cdot \zeta_1 a,
\end{split}
\end{equation}
provided that  the amplitude $a\in C^\infty(\R^n)$ satisfies the regularized transport equation 
\begin{equation}
\label{firsttransport}
-2\zeta_0 \cdot \nabla a + \zeta_0 \cdot V_\tau a = 0. 
\end{equation} 
One looks for a solution to \eqref{firsttransport} in the form $a=e^{\Phi_\tau}$, where $\Phi_\tau$ solves the equation
\begin{equation}
\label{CGO3}
-2\zeta_0\cdot\nabla \Phi_\tau + \zeta_0\cdot V_\tau=0 \quad \textrm{in}\quad \R^n. 
\end{equation}
Using that $\zeta\cdot\zeta=0$ where $\zeta=\zeta_0+\mathcal{O}(h)$, we have  $\zeta_0\cdot\zeta_0=0$ and $|\Re\zeta_0|=|\Im\zeta_0|=1$.   The operator   $N_{\zeta_0}:=\zeta_0\cdot\nabla$ is  therefore the $\bar\p$--operator in suitable linear coordinates. An inverse  of this operator is defined by
\[
(N_{\zeta_0}^{-1}f)(x) =\frac{1}{2\pi}\int_{\R^2} \frac{f(x-y_1\Re\zeta_0 -y_2\Im\zeta_0)}{y_1+iy_2}dy_1dy_2,\quad f\in C_0(\R^n).
\]
We have the following result from \cite[Lemma 4.6]{Salo_diss}. 
\begin{lem}
\label{lem_salo_1}
Let $f\in W^{k,\infty}(\R^n)$, $k\ge 0$, with $\emph{\supp}(f)\subset B(0,R)$. Then $\Phi=N_{\zeta_0}^{-1} f\in W^{k,\infty}(\R^n)$ satisfies $N_{\zeta_0}\Phi=f$ in $\R^n$, and we have 
\begin{equation}
\label{eq_salo_1}
\|\Phi\|_{W^{k,\infty}(\R^n)}\le C\|f\|_{W^{k,\infty}(\R^n)}, 
\end{equation}
where $C=C(R)$.  If $f\in C_0(\R^n)$, then $\Phi\in C(\R^n)$. 
\end{lem}
By Lemma \ref{lem_salo_1}, we see that  $\Phi_\tau(x,\zeta_0; \tau):=\frac{1}{2}N_{\zeta_0}^{-1}(\zeta_0\cdot V_\tau)\in (C^\infty \cap L^\infty)(\R^n)$ satisfies  \eqref{CGO3}. Thus, by Proposition \ref{approximationestimates} and Lemma \ref{lem_salo_1}, we have
\begin{equation}
\label{phiestimate1}
\|\Phi_\tau\|_{L^\infty(\mathbb{R}^n)} = \mathcal{O}(1), \quad \|\nabla \Phi_\tau\|_{L^\infty(\mathbb{R}^n)} = \mathcal{O}(\tau^{-1}) \quad \textrm{as} \quad  \tau\rightarrow 0.
\end{equation}
	
We now need the following result established in \cite[Lemma 3.1]{Syl_Uhl_1987}:
\begin{lem}
\label{lem_SU}
Let $-1<\delta<0$ and let $f\in L^2_{\delta+1}(\R^n)$. Then there exists a constant $C>0$, independent of $\zeta_0$, such that 
\[
\|N^{-1}_{\zeta_0} f\|_{L^2_\delta(\R^n)}\le C\|f\|_{L^2_{\delta+1}(\R^n)}.
\]
\end{lem} 
Here, 
\[
\|f\|_{L^2_\delta(\R^n)}^2=\int_{\R^n}(1+|x|^2)^{\delta}|f(x)|^2dx. 
\]
Using the fact that $V_\tau$ has compact support uniformly in $\tau$, by Lemma \ref{lem_SU} and Proposition \ref{approximationestimates}, we have \begin{equation}
\label{phiestimate2}
 \|\nabla \Phi_\tau\|_{L^2(\Omega)} = \mathcal{O}(1), \quad \|\p^\alpha \Phi_\tau\|_{L^2(\Omega)} = o(\tau^{-1}), \quad |\alpha|=2,\quad \tau\rightarrow 0.
\end{equation}	
Setting $\Phi(\cdot,\zeta_0):=N_{\zeta_0}^{-1}(\frac{1}{2}\zeta_0\cdot V)\in L^\infty(\R^n)$, it follows from Proposition \ref{approximationestimates} and Lemma \ref{lem_SU} that $\|\Phi - \Phi_\tau\|_{L^2(\Omega)} = o(\tau)$ as $\tau\to 0$. 
	
Recalling that $a = e^{\Phi_\tau}$, and using \eqref{phiestimate1} and \eqref{phiestimate2}, we get 
\begin{equation}
\label{3.12}
\|a\|_{L^\infty(\Omega)} = \mathcal{O}(1), \quad \|\nabla a\|_{L^\infty(\Omega)} = \mathcal{O}(\tau ^{-1}),\quad \|\nabla a\|_{L^2(\Omega)} = \mathcal{O}(1).
\end{equation}  
Furthermore, we have $\Delta a=a\nabla\Phi_\tau\cdot \nabla\Phi_\tau +a\Delta \Phi_\tau$. Using  \eqref{phiestimate1}, \eqref{phiestimate2}, \eqref{3.12}, and the Gagliardo-Nirenberg inequality for $u \in (L^\infty \cap H^2)(\Omega)$, see \cite[page 313]{Brezisbook}, \begin{equation}
\|\nabla u\|_{L^4(\Omega)} \le C\|u\|_{L^\infty(\Omega)}^{1/2}\| u\|_{H^2(\Omega)}^{1/2},
\end{equation}  
with $u=\Phi_\tau$,  we obtain that 
\begin{equation}
\label{star}
\|\Delta a\|_{L^2(\Omega)} \le \|a\|_{L^\infty(\Omega)}(\|\nabla \Phi_\tau\|_{L^4(\Omega)}^2+\|\Delta \Phi_\tau\|_{L^2(\Omega)})= o(\tau ^{-1}).
\end{equation} 
	
The discussion in \cite[Section 3]{KKGU2017geo} can now be applied exactly as it stands, and we obtain that the norm of  the right hand side of (\ref{CGO2}) in $H^{-1}_{\textrm{scl}}(\Omega)$ does not exceed $h^2o(\tau^{-1}) + ho(\tau)$ as $\tau \rightarrow 0$. Choosing now $\tau=h^{1/2}$, applying Proposition \ref{solvability result} to (\ref{CGO2}) with $\alpha=\Re \zeta_0$ and using the fact that $e^{\frac{x\cdot\zeta_1}{h}}=\mathcal{O}(1)$ with all derivatives and  $(hD_x)^\alpha (e^{\frac{ix\cdot \text{Im} \zeta_0}{h}})=\mathcal{O}(1)$ for all $\alpha$, we see that there exists a solution $r\in H^1(\Omega)$ of (\ref{CGO2}) such that $\|r\|_{H^1_{\textrm{scl}}(\Omega)}=o(h^{1/2})$ as $h\to 0$. 
	
We summarize the discussion above in the following  proposition. 
\begin{prop}
\label{prop_cgo_solutions}
Let $V,W\in (H^1\cap L^\infty)(\Omega; \C^n)$, $q\in L^\infty(\Omega; \C)$, and let $\zeta\in \C^n$ be such that $\zeta\cdot\zeta=0$, $\zeta = \zeta_0 + \zeta_1$ with $\zeta_0$ being independent of $h>0$,  $|\emph{\text{Re}}\,\zeta_0|=|\emph{\text{Im}}\, \zeta_0|=1$, and $\zeta_1 = \mathcal{O}(h)$ as $h\rightarrow 0$. Then for all $h>0$ small enough, there exists a solution $u(x,\zeta;h)\in H^1(\Omega)$ to the equation $P_{V,W,q}u=0$ in $\Omega$, of the form
\[
u(x,\zeta;h)=e^{\frac{x\cdot\zeta}{h}}(a(x,\zeta_0;h)+r(x,\zeta;h)),
\] 
where $a=e^{\Phi_h(x,\zeta_0)}$ with  $\Phi_h(\cdot,\zeta_0)\in (C^\infty \cap L^\infty)(\R^n)$. We have  $\|\Phi - \Phi_h\|_{L^2(\Omega)} = o(h^{1/2})$ as $h\to 0$ where  $\Phi(\cdot,\zeta_0):=N_{\zeta_0}^{-1}(\frac{1}{2}\zeta_0\cdot V)\in L^\infty(\R^n)$. 
Moreover, $a$ satisfies 
\begin{equation}
\label{3.12_new}
\begin{aligned}
&\|a\|_{L^\infty(\Omega)} = \mathcal{O}(1), \quad \|\nabla a\|_{L^\infty(\Omega)} = \mathcal{O}(h ^{-\frac{1}{2}}),\quad \|\nabla a\|_{L^2(\Omega)} = \mathcal{O}(1), \\
 &\|\Delta a\|_{L^2(\Omega)} = o(h ^{-\frac{1}{2}}).
 \end{aligned}
\end{equation} 
The remainder $r$ is such that 
$ \|r\|_{H^1_{\text{scl}}(\Omega)}=o(h^{\frac{1}{2}})$,  $h\to 0$.
\end{prop}

\section{Proof of Theorem \ref{thm_main_complex} and Corollary \ref{thm_main_real}}
	
\label{sec_proof_main}
Let $A_1, A_2 \in (H^1\cap L^{\infty})(\Omega; \mathbb{C}^n)$. Using a standard Seeley extension argument, we observe that the vector fields $A_1$ and $A_2$ can be extended to elements of  $(H^1\cap L^\infty\cap\mathcal{E}')(\R^n; \C^n)$, see 
\cite[Section 2.2]{KKGU2016} for a detailed discussion. Letting $v:=(A_1-A_2)1_{\R^n\setminus\overline{\Omega}}\in (\mathcal{E}'\cap L^\infty)(\R^n;\C^n)$, we see that $v\in H^1(\R^n; \C^n)$ in view of the fact that $\Lambda_{A_1, q_1} = \Lambda_{A_2, q_2}$ together with   Proposition \ref{boundaryresult}. Here $1_{\R^n\setminus\overline{\Omega}}$ is the characteristic function of the set $\R^n\setminus\overline{\Omega}$. Replacing $A_2$ by $A_2+v$, we achieve that $A_1=A_2$ on $\R^n\setminus \overline{\Omega}$. We also extend $q_1$ and $q_2$ to all of $\R^n$ so that $q_1=q_2=0$ on  $\R^n\setminus \Omega$. 

Let $B\subset \R^n$ be a large open ball such that $\Omega\subset\subset B$.  Using that $\Lambda_{A_1, q_1} = \Lambda_{A_2, q_2}$ and $A_1=A_2$, $q_1=q_2$ on $B\setminus \Omega$,  similarly to  \cite[Proposition 3.4]{KKGU2014}, we conclude that  $C_{A_1, q_1}(B)=C_{A_2, q_2}(B)$, where 
\[
C_{A_j, q_j}(B)=\{(u|_{\p B},\p_\nu u|_{\p B}): u\in H^1(B), \, L_{A_j, q_j}u=0 \text{ in } B \}
\]
is the set of the Cauchy data for the operator $L_{A_j,q_j}$ in $B$.  Using that $C_{A_1, q_1}(B)=C_{A_2, q_2}(B)$, similarly to the proof of \cite[Proposition 4.1]{KKGU2017geo}  we then derive the following integral identity,  
\begin{equation} 
\label{intidentity2}
\int_B [(2i(A_2-A_1)\cdot \nabla u_1)u_2 + (q_1 - q_2)u_1u_2]dx = 0
\end{equation}  
for all $u_1, u_2 \in H^1(B)$ solving 
\begin{equation}
\label{eq_2_1new}
-\Delta u_1-2iA_1\cdot \nabla u_1+q_1u_1=0\quad \textrm{in}\quad B,
\end{equation}
\begin{equation}
\label{eq_2_2new}
-\Delta u_2+2iA_2\cdot \nabla u_2+2i(\nabla \cdot A_2)u_2+q_2u_2=0\quad \textrm{in}\quad B.
\end{equation}

We shall use \eqref{intidentity2} with $u_1$ and $u_2$ being CGO solutions to  \eqref{eq_2_1new}  and \eqref{eq_2_2new}. To this end,  let $\xi,\mu_1,\mu_2\in\R^n$ be such that $|\mu_1|=|\mu_2|=1$ and $\mu_1\cdot\mu_2=\mu_1\cdot\xi=\mu_2\cdot\xi=0$. Similarly to \cite{KKGU2014} and \cite{Sun_1993}, we set 
\begin{equation}
\label{eq_zeta_1_2}
\zeta_1=\frac{ih\xi}{2}+\mu_1 + i\sqrt{1-h^2\frac{|\xi|^2}{4}}\mu_2 , \quad 
\zeta_2=\frac{ih\xi}{2}-\mu_1-i\sqrt{1-h^2\frac{|\xi|^2}{4}}\mu_2,
\end{equation}
so that $\zeta_j\cdot\zeta_j=0$, $j=1,2$, and $\frac{\zeta_1+\zeta_2}{h}=i\xi$.   Moreover, $\zeta_1= \mu_1+ i\mu_2+\mathcal{O}(h)$ and $\zeta_2= -\mu_1- i\mu_2+\mathcal{O}(h)$ as $h\to 0$.  By Proposition \ref{prop_cgo_solutions},  for all $h>0$ small enough, there exists a solution $u_1(x,\zeta_1;h)\in H^1(B)$ to \eqref{eq_2_1new} of the form
\begin{equation}
\label{eq_u_1}
u_1(x,\zeta_1;h)=e^{\frac{x\cdot\zeta_1}{h}}(e^{\Phi_{1,h}(x,\mu_1+i\mu_2)}+r_1(x,\zeta_1;h)),
\end{equation} 
and 
$u_2(x,\zeta_2;h)\in H^1(B)$ to \eqref{eq_2_2new}  of the form
\begin{equation}
\label{eq_u_2}
u_2(x,\zeta_2;h)=e^{\frac{x\cdot\zeta_2}{h}}(e^{\Phi_{2,h}(x,-\mu_1-i\mu_2)}+r_2(x,\zeta_2;h)),
\end{equation}
where $\Phi_{j,h}\in (C^\infty \cap L^\infty)(\R^n)$, and 
\begin{equation}
\label{5.18}
\|\Phi_j- \Phi_{j,h}\|_{L^2(B)}=o(h^{\frac{1}{2}}),\quad  h \to 0, \quad  j = 1,2.
\end{equation}
Here
\begin{equation}
\label{eq_phi_1_def}
\Phi_{1}(\cdot,\mu_1+i\mu_2):=N_{\mu_1+i\mu_2}^{-1}
((\mu_1+i\mu_2)\cdot (-iA_1))\in L^\infty(\R^n),
\end{equation}
and 
\begin{equation}
\label{eq_phi_2_def}
\Phi_{2}(\cdot,-\mu_1-i\mu_2):=N_{-\mu_1-i\mu_2}^{-1}
((-\mu_1-i\mu_2)\cdot (iA_2))\in L^\infty(\R^n).
\end{equation}
The remainder $r_j$ is such that
\begin{equation}
\label{5.22}
\|r_j\|_{H^1_{\text{scl}}(B)}=o(h^{\frac{1}{2}}), \quad  h\to 0, \quad j=1,2.
\end{equation} 
We next substitute the CGO solutions $u_1$ and $u_2$, given by \eqref{eq_u_1} and \eqref{eq_u_2}, into the integral identity  \eqref{intidentity2}, multiply it by $h$, and let $h\to 0$. 
First we have
\begin{equation}
\label{hnablau1}
h(\nabla u_1)u_2 = \zeta_1 e^{ix \cdot \xi}(a_1+r_1)(a_2+r_2) + he^{ix \cdot \xi}(\nabla a_1 + \nabla r_1)(a_2+r_2),
\end{equation}
where  $a_j=e^{\Phi_{j,h}}$, $j=1,2$.

Using \eqref{5.22} and \eqref{3.12_new}, we obtain that
\begin{equation}
\label{eq_200_1}
\bigg|\int_B 2i(A_2-A_1)\cdot \zeta_1 e^{ix\cdot \xi}(a_1r_2+r_1a_2+r_1r_2)dx\bigg| = o(h^\frac{1}{2}) \to 0,  \quad h \to  0,
\end{equation}
\begin{equation}
\label{eq_200_2}
\bigg|h\int_B 2ie^{ix\cdot \xi}(A_2-A_1) \cdot (\nabla a_1 +\nabla r_1)(a_2+r_2)dx\bigg| = o(h^\frac{1}{2}), \quad h \to  0,
\end{equation}
\begin{equation}
\label{eq_200_3}
\bigg|h\int_B (q_2-q_1)e^{ix \cdot \xi} (a_1 +r_1)(a_2+r_2)dx\bigg| = \mathcal{O}(h),  \quad h \to  0.
\end{equation} 
Therefore, from \eqref{intidentity2} in view of \eqref{eq_200_1}, \eqref{eq_200_2}, and \eqref{eq_200_3}, we get  
\[
{\lim_{h \to 0}} \int_B (A_2-A_1)\cdot (\mu_1+i\mu_2)e^{ix\cdot \xi}e^{\Phi_{1,h}+\Phi_{2,h}}dx =0.
\]
Using \eqref{5.18}, similarly to \cite{KKGU2014},  we obtain that  
\begin{equation}
\label{A1=A2}
 (\mu_1+i\mu_2) \cdot \int_{\R^n} (A_2-A_1)e^{ix\cdot \xi}e^{\Phi_1+\Phi_2}dx=0. 
\end{equation} 
Here the integration is extended to all of $\R^n$ since $\supp(A_2-A_1)\subset B$.

In view of $N_{-\zeta}^{-1}f=-N_\zeta^{-1}f$, we have 
\begin{equation}
\label{eq_sum_phase}
\Phi_{1}+\Phi_{2}=N_{\mu_1+i\mu_2}^{-1}(- i(\mu_1+i\mu_2)\cdot (A_1-A_2)).  
\end{equation}
An application of \cite[Proposition 3.3]{KKGU2014} allows us to get
\begin{equation}
\label{eq_without_phases}
(\mu_1+i\mu_2)\cdot\int_{\R^n} (A_2-A_1) e^{ix\cdot\xi} dx=0, 
\end{equation} 
and we conclude as in  \cite{KKGU2014} that 
\begin{equation}
\label{eq_rec_dA}
d(A_2-A_1)=0\quad \text{in}\quad  \R^n,
\end{equation}
 see also \cite{EskinRal}. Here we view the vector field $A:=A_2-A_1$ as a one form and $dA$ is a two form given by
\[
dA=\sum_{1\le j<k\le n} (\p_{x_j}A_k-\p_{x_k}A_j)dx_j\wedge dx_k.
\]

Our next goal is to recover $A$ and $q$ up to a gauge transformation. To this end, similarly to  \cite{KKGU2014} we observe that \eqref{eq_rec_dA} implies that there exists 
\begin{equation}
\label{eq_func_varphi_referee}
\varphi\in (W^{1,\infty}\cap H^2)(\R^n)\text{  with }\supp(\varphi)\subset B
\end{equation}
such that 
$\nabla \varphi=A_1-A_2$.  It follows that 
\begin{equation}
\label{5.31}
\begin{split}
e^{-i\varphi} \circ (-\Delta -2iA_2\cdot \nabla+ q_2) \circ e^{i\varphi}
&= -\Delta -2i A_1 \cdot \nabla +\nabla  \cdot (-i\nabla \varphi) +\tilde{q_2} = P_{V,W, \tilde{q_2}},
\end{split}
\end{equation}
where
\begin{equation}
\label{5.32}
\tilde{q_2} = q_2 + 2A_2 \cdot \nabla \varphi +(\nabla \varphi)^2 \in L^\infty(\R^n;\C),
\end{equation} 
\begin{align*}
V= -2iA_1\in (H^1\cap L^\infty)(\R^n;\C^n),\quad 
W=-i\nabla \varphi\in  (H^1\cap L^\infty)(\R^n;\C^n).
\end{align*}
Associated to the operator $P_{V,W, \tilde{q_2}}$ is the set of the Cauchy data,
\[
C_{V,W,\tilde q_2}(B)=\{(u|_{\p B}, \p_\nu u|_{\p B}): u\in H^1(B), \ P_{V,W,\tilde{q_2}}u=0\}.
\] 
Using \eqref{5.31} and  arguing as in  \cite{KKGU2017geo}, we see that 
\[
C_{V,W,\tilde q_2}(B)=C_{A_2,q_2}(B), 
\]
and therefore,  $C_{V,W,\tilde q_2}(B)=C_{A_1,q_1}(B)$.  Similarly to  \cite{KKGU2017geo}  this equality of the sets of the Cauchy data  implies that the following integral identity holds, 
\begin{equation}
\label{2ndintidentity}
\int_B [(\tilde{q_2} - q_1)u_1u_2 +i\nabla \varphi \cdot \nabla (u_1u_2)]dx = 0,
\end{equation}  
for all $u_1, u_2 \in H^1(B)$ solving 
\begin{equation}
\label{2ndintidentity1}
-\Delta u_1-2iA_1\cdot \nabla u_1+q_1u_1=0\quad \textrm{in}\quad B,
\end{equation}
\begin{equation}
\label{2ndintidentity2}
-\Delta u_2+2iA_1\cdot \nabla u_2+2i(\nabla \cdot A_1)u_2 + \nabla \cdot (-i\nabla \varphi)u_2+ \tilde{q_2}u_2=0\quad \textrm{in}\quad B.
\end{equation}
	
The next step is to substitute CGO solutions to \eqref{2ndintidentity1} and \eqref{2ndintidentity2} into the integral identity \eqref{2ndintidentity}. To this end, let $\zeta_1, \zeta_2 \in \C^n$ be given by \eqref{eq_zeta_1_2} and let us recall from Proposition \ref{prop_cgo_solutions} that  for all $h>0$ small enough, there exists a solution $u_1(x,\zeta_1;h)\in H^1(B)$ to \eqref{2ndintidentity1} of the form
\begin{equation}
\label{u12ndsolution}
u_1(x,\zeta_1;h)=e^{\frac{x\cdot\zeta_1}{h}}(e^{\Phi_{1,h}(x,\mu_1+i\mu_2)}+r_1(x,\zeta_1;h)),
\end{equation} 
and $u_2(x,\zeta_2;h)\in H^1(B)$ to \eqref{2ndintidentity2} of the form
\begin{equation}
\label{u22ndsolution}
u_2(x,\zeta_2;h)=e^{\frac{x\cdot\zeta_2}{h}}(e^{\Phi_{2,h}(x,-\mu_1-i\mu_2)}+r_2(x,\zeta_2;h)),
\end{equation} 
where 
\begin{equation}
\label{eq_phi_1_2nddef}
\Phi_{1,h}(\cdot,\mu_1+i\mu_2):=N_{\mu_1+i\mu_2}^{-1}((\mu_1+i\mu_2)\cdot (-iA_{1,h})),
\end{equation}
and 
\begin{equation}
\label{eq_phi_2_2nddef}
\Phi_{2,h}(\cdot,-\mu_1-i\mu_2):=N_{-\mu_1-i\mu_2}^{-1}
((-\mu_1-i\mu_2)\cdot (iA_{1,h})).
\end{equation}
Here  $A_{1,h}$ is the regularization of $ A_1$ as above. Notice that $\Phi_{1,h}(\cdot,\mu_1+i\mu_2)+ \Phi_{2,h}(\cdot,-\mu_1-i\mu_2)=0$. Let us also recall that the remainders $r_j$ satisfy \eqref{5.22}. 

Letting $a_j = e^{\Phi_{j,h}}, j=1,2$, so that $a_1a_2=1$,  we write 
\begin{align*}
&u_1u_2=e^{ix\cdot \xi} (1+a_1r_2 + r_1a_2+r_1r_2), \\
&\nabla(u_1u_2)=i\xi e^{ix\cdot\xi}(1+a_1r_2 + r_1a_2+r_1r_2)+ e^{ix\cdot \xi}\nabla(a_1r_2 + r_1a_2+r_1r_2).
\end{align*}
It follows from \eqref{3.12_new} and \eqref{5.22} that 
\begin{equation}
\label{eq_500_1}
\int_B (\tilde q_2-q_1)u_1u_2dx\to \int_B (\tilde q_2-q_1)e^{ix\cdot\xi}dx, \quad h\to 0,
\end{equation}
\begin{equation}
\label{eq_500_2}
\int_B i\nabla \varphi \cdot i\xi e^{ix\cdot\xi}(1+a_1r_2 + r_1a_2+r_1r_2) dx\to \int_B i\nabla \varphi \cdot i\xi e^{ix\cdot\xi}dx, \quad h\to 0. 
\end{equation}
Using \eqref{5.22}, we also have
\begin{equation}
\label{eq_500_3}
\bigg|\int_B i\nabla \varphi \cdot e^{ix\cdot\xi} \nabla (r_1r_2)dx\bigg|\le C \|\nabla \varphi\|_{L^\infty}(\|\nabla r_1\|_{L^2}\|r_2\|_{L^2}+\|\nabla r_2\|_{L^2}\|r_1\|_{L^2})=o(1),
\end{equation}
as $h\to 0$.
Finally we claim that 
\begin{equation}
\label{eq_500_4}
\bigg|\int_B i\nabla \varphi \cdot e^{ix\cdot\xi} \nabla (a_1r_2+a_2r_1)dx\bigg|=o(h^{\frac{1}{2}}), \quad h\to 0. 
\end{equation}
Following \cite{Pohjola}, \cite{KKGU2017geo}, when establishing \eqref{eq_500_4}  we introduce the regularization $\varphi_\tau=\varphi* \psi_\tau\in C^\infty_0(\R^n)$, $\tau>0$. Here $\psi _\tau (x) = \tau ^{-n}\psi\bigg(\dfrac{x}{\tau}\bigg)$ with $\psi \in C_0^\infty(\mathbb{R}^n), 0\le \psi \le 1$, and $\int_{\mathbb{R}^n} \psi dx =1$. Assume also that $\psi$  is radial.  We have 
$\supp(\varphi_\tau)\subset B$, for all $\tau>0$ small enough.  Using that $\nabla \varphi\in (H^1\cap L^\infty)(\R^n)$ and Proposition \ref{approximationestimates}, we get 
\begin{equation}
\label{eq_500_5}
\|\nabla \varphi -\nabla \varphi_\tau\|_{L^2} = o(\tau),  \quad \|\Delta\varphi_\tau\|_{L^2}=\mathcal{O}(1), \quad \tau\to 0. 
\end{equation}
In view of \eqref{eq_500_5},  \eqref{3.12_new} and \eqref{5.22}, we obtain that 
\begin{align*}
\bigg|\int_B i\nabla \varphi \cdot e^{ix\cdot\xi} \nabla (a_1r_2)dx\bigg|\le \int_B | (\nabla \varphi -\nabla\varphi_\tau) \cdot  \nabla (a_1r_2)|dx + \int_B |(\Delta\varphi_\tau) a_1r_2|dx\\
=o(\tau)o(h^{-\frac{1}{2}})+ o(h^{\frac{1}{2}})=o(h^{\frac{1}{2}}), \quad h\to 0, 
\end{align*}
where we take $\tau=h$. The estimate \eqref{eq_500_4} follows.

Combining \eqref{2ndintidentity}, \eqref{eq_500_1},  \eqref{eq_500_2},  \eqref{eq_500_3}, and  \eqref{eq_500_4}, we get 
\[
\int_B \big( (\tilde q_2-q_1) +  i\nabla \varphi \cdot i\xi \big) e^{ix\cdot\xi}dx=0. 
\]
In other words, $\mathcal{F}(\tilde{q_2}-q_1 - i\Delta \varphi) = 0$ in the sense of distributions, and therefore, 
\begin{equation}
\label{eq_func_varphi_referee_2}
\tilde{q_2}-q_1 - i\Delta \varphi = 0 \quad  \text{in}\quad  \R^n. 
\end{equation}
In the view of \eqref{5.32}, the proof of Theorem \ref{thm_main_complex} is now complete. 

It follows from the proof of Theorem \ref{thm_main_complex}, in particular from \eqref{eq_func_varphi_referee} and \eqref{eq_func_varphi_referee_2},   in view of the fact that the vector fields $A_1$, $A_2$, and the scalar potentials $q_1$, $q_2$ are real-valued, that $\varphi\in (W^{1,\infty}\cap H^2)(\R^n; \R)$ with $\supp(\varphi)\subset B$, is such that  $\Delta\varphi=0$ in $\R^n$. Hence, $\varphi=0$. This completes the proof of Corollary \ref{thm_main_real}.

\section{Proofs of Corollary \ref{twofrequency}  and Corollary \ref{threefrequency}}
\label{tomography}

The goal of this section is to prove Corollary \ref{twofrequency}  and Corollary \ref{threefrequency}, by following the arguments of the paper \cite{AARN2015}, and verifying that they still go through in the present low regularity setting, once Theorem \ref{thm_main_complex}  and the boundary reconstruction result of Proposition \ref{boundaryresult} have been established. The following discussion is therefore provided mainly for the convenience of the reader.

Let the fluid parameters $c_j$, $\rho_j$, $v_j$ and $\alpha_j$ satisfy \eqref{eq_coeff_our}, and let us define $A_j(\omega)$ and $q_j(\omega)$ as in \eqref{eq_int_acoustic_tom_coeff}, $j=1,2$.  Using that $\Lambda_{A_1(\omega),q_1(\omega)}=\Lambda_{A_2(\omega), q_2(\omega)}$ for  $\omega=\omega_1$ and $\omega=\omega_2$, we conclude from Theorem \ref{thm_main_complex} that there is $\varphi(\omega)\in W^{1,\infty}(\Omega;\C)$ such that 
 \[
 A_1(\omega)=A_2(\omega)+\nabla \varphi(\omega), \quad q_1(\omega)=q_2(\omega)+2A_2(\omega)\cdot \nabla \varphi(\omega)
 +(\nabla \varphi(\omega))^2-i\Delta \varphi (\omega),
 \]
for $\omega=\omega_1, \omega_2$. Hence, 
\[
q_1(\omega)=q_2(\omega)+2A_2(\omega)\cdot  ( A_1(\omega)-A_2(\omega))
 +( A_1(\omega)-A_2(\omega))^2-i\nabla\cdot  ( A_1(\omega)-A_2(\omega)),
\]
and therefore, 
\begin{equation}
\label{eq 6.1}
q_2(\omega) - q_1(\omega) +A_1(\omega)^2-A_2(\omega)^2 -i\nabla \cdot (A_1(\omega) - A_2(\omega)) = 0,
\end{equation} 
for $\omega=\omega_1$ and $\omega=\omega_2$.
Taking the real and imaginary parts in \eqref{eq 6.1},  using \eqref{eq_int_acoustic_tom_coeff} and the following consequence of it,
\[
A_j(\omega)^2 = \omega ^2 \frac{v_j^2}{c_j^4} -\bigg(\frac{1}{2}\frac{\nabla \rho_j}{\rho_j}\bigg)^2+i\omega \frac{v_j}{c_j^2}\cdot \frac{\nabla \rho_j}{\rho_j}, \quad  j = 1,2, 
\]
we get 
\begin{equation}
\label{eq 6.3}
-\omega^2\bigg(\frac{1}{c_2^2} - \frac{1}{c_1^2}\bigg) + \omega ^2 \bigg(\frac{v_1^2}{c_1^4} - \frac{v_2^2}{c_2^4}\bigg) - \bigg(\frac{1}{2}\frac{\nabla \rho_1}{\rho_1}\bigg)^2 + \bigg(\frac{1}{2}\frac{\nabla \rho_2}{\rho_2}\bigg)^2 + \nabla \cdot \bigg(\frac{\nabla \rho_1}{2\rho_1} - \frac{\nabla \rho_2}{2\rho_2} \bigg)= 0,
\end{equation}
and 
\begin{equation}
\label{eq 6.4}	
\omega\bigg(\frac{v_1}{c_1^2} \cdot \frac{\nabla \rho_1}{\rho_1} - \frac{v_2}{c_2^2} \cdot \frac{\nabla \rho_2}{\rho_2}\bigg) - \omega \nabla \cdot \bigg(\frac{v_1}{c_1^2} - \frac{v_2}{c_2^2}\bigg)-2\omega \bigg(\frac{\alpha_2(\omega)}{c_2}-\frac{\alpha_1(\omega)}{c_1}\bigg) = 0,
\end{equation}
for $\omega=\omega_1$ and $\omega=\omega_2$. Using that $\omega_1\ne \omega_2$,  we obtain from \eqref{eq 6.3}  that 
\begin{equation}
\label{eq 6.5}
 \frac{v_1^2}{c_1^4} - \frac{v_2^2}{c_2^4} - \bigg(\frac{1}{c_2^2} - \frac{1}{c_1^2}\bigg) = 0, 
\end{equation}
\begin{equation}
\label{eq 6.6}
\nabla \cdot \bigg(\frac{\nabla \rho_1}{2\rho_1} - \frac{\nabla \rho_2}{2\rho_2} \bigg) - \bigg(\frac{1}{2}\frac{\nabla \rho_1}{\rho_1}\bigg)^2 + \bigg(\frac{1}{2}\frac{\nabla \rho_2}{\rho_2}\bigg)^2= 0, 
\end{equation}
and from  \eqref{eq 6.4} that
\begin{equation}
\label{eq 6.7}
\frac{v_1}{c_1^2} \cdot \frac{\nabla \rho_1}{\rho_1} - \frac{v_2}{c_2^2} \cdot \frac{\nabla \rho_2}{\rho_2} - \nabla \cdot \bigg(\frac{v_1}{c_1^2} - \frac{v_2}{c_2^2}\bigg) -2 \bigg(\frac{\alpha_2(\omega)}{c_2}-\frac{\alpha_1(\omega)}{c_1}\bigg)  = 0.
\end{equation}

Proposition \ref{boundaryresult} gives 
\begin{equation}
\label{eq_100_1}
\frac{\nabla \rho_1}{\rho_1}\bigg|_{\partial \Omega} = \frac{\nabla \rho_2}{\rho_2}\bigg|_{\partial \Omega}\quad \text{in}\quad H^{\frac{1}{2}}(\p \Omega),
\end{equation}
and 
\begin{equation}
\label{eq_100_2}
\frac{v_1}{c_1^2}\bigg|_{\partial \Omega} = \frac{v_2}{c_2^2}\bigg|_{\partial \Omega}\quad \text{in}\quad H^{\frac{1}{2}}(\p \Omega).
\end{equation}
Letting $u_j = \frac{1}{2} \log \rho_j\in (W^{1,\infty}\cap H^2)(\Omega; \R)$, and   using  \eqref{eq_100_1} and the connectedness of  $\partial \Omega$, we see that $g=u_1-u_2$ is a constant along $\p \Omega$. Furthermore, letting $X = \nabla u_1 + \nabla u_2\in   (L^{\infty}\cap H^1)(\Omega; \R^n)$, and using \eqref{eq 6.6}, we get 
\[
\Delta g - X\cdot \nabla g = 0 \quad \text{in} \quad  \Omega.
\]
An application of the  maximum principle  gives that $g$ is a constant in $\Omega$, see \cite[Chapter 3, Section 8.2]{Aubin_book}. Hence,  $\rho_1 = C\rho_2$ in $\Omega$. 

\textbf{Conclusion of the proof of Corollary \ref{twofrequency}.} Let $\omega=\omega_1$. Taking the real part of $A_1(\omega)-A_2(\omega)=\nabla \varphi(\omega)$, we see that 
\begin{equation}
\label{eq 6.12}
 \frac{v_1}{c_1^2} -  \frac{v_2}{c_2^2} = \nabla \chi, \quad  \chi=\frac{\Re \varphi(\omega)}{\omega}.
\end{equation} 
Setting  $a = \frac{\nabla \rho_1}{\rho_1} = \frac{\nabla \rho_2}{\rho_2}\in   (L^{\infty}\cap H^1)(\Omega; \R^n)$, and recalling that $\alpha_1(\omega)=\alpha_2(\omega)=0$,  we obtain from \eqref{eq 6.7} and \eqref{eq_100_2} that 
\begin{align*}
&a\cdot \nabla \chi-\Delta\chi= 0 \quad \text{in}\quad \Omega,\\
&\chi=B  \quad \text{on}\quad \p \Omega,
\end{align*}
where $B$ is a constant.  Another application of the maximum principle gives that $\chi=B$ in  $\Omega$, and therefore, $\frac{v_1}{c_1^2} -  \frac{v_2}{c_2^2} = 0$ in $\Omega$. Now \eqref{eq 6.5} implies that $c_1 = c_2$ in $\Omega$, and thus,  $v_1 = v_2$ in $\Omega$. This completes the proof of Corollary \ref{twofrequency}.

\textbf{Proof of Corollary  \ref{threefrequency}.}  It follows from \eqref{eq 6.7} that 
\begin{equation}
\label{eq_100_4}
\bigg(\frac{v_1}{c_1^2}  - \frac{v_2}{c_2^2} \bigg)\cdot a - \nabla \cdot \bigg(\frac{v_1}{c_1^2} - \frac{v_2}{c_2^2}\bigg) +2 \omega^{\zeta_1} 
\frac{\alpha_{0,1}}{c_2}- 2 \omega^{\zeta_2}\frac{\alpha_{0,2}}{c_2}  = 0,
\end{equation}
for $\omega=\omega_1,\omega_2,\omega_3>0$ mutually different frequencies. If $\zeta_1(x)\ne \zeta_2(x)$ then the vectors $(1,\omega^{\zeta_1(x)}, \omega^{\zeta_2(x)})$, $\omega=\omega_1,\omega_2,\omega_3$, are linearly independent in $\R^3$.  Hence, \eqref{eq_100_4} implies that at the point $x$, we have
\begin{equation}
\label{eq_100_5}
\bigg(\frac{v_1}{c_1^2}  - \frac{v_2}{c_2^2} \bigg)\cdot a  - \nabla \cdot \bigg(\frac{v_1}{c_1^2} - \frac{v_2}{c_2^2}\bigg)  = 0, 
\end{equation}
$\frac{\alpha_{0,1}}{c_2} =0$, $ \frac{\alpha_{0,2}}{c_2}=0$.
If $\zeta_1(x)=\zeta_2(x)$ then the vectors $(1,\omega^{\zeta_1(x)})$, $\omega=\omega_1,\omega_2$, are linearly independent and at the point $x$, \eqref{eq_100_4} gives that \eqref{eq_100_5} holds and 
$\frac{\alpha_{0,1}}{c_2} - \frac{\alpha_{0,2}}{c_2}=0$.
As in the proof of Corollary \ref{twofrequency}, we get $c_1 = c_2$, and $v_1 = v_2$ in $\Omega$.  Furthermore, $\alpha_1(x)=\alpha_2(x)$. This completes the proof of Corollary  \ref{threefrequency}.

\appendix
\section{Boundary reconstruction of ($H^1 \cap L^\infty$)-vector field}
\label{boundary}
The purpose of this appendix is to provide a proof of the boundary reconstruction of the $(H^1 \cap L^\infty)$-vector field $A$ from the knowledge of the Dirichlet-to-Neumann map $\Lambda_{A,q}$ for the operator $L_{A,q}= -\Delta - 2iA\cdot \nabla  + q$. When doing so, we follow the arguments of \cite[Appendix A]{KKGU2017geo} closely, the only difference being that here the potential $q$ is present, whereas it was absent in \cite{KKGU2017geo}. We refer to  \cite{Brown2001}, \cite{Brownsalo}, \cite{KKGU2017geo}, and \cite{KKGU2017} for similar reconstruction arguments. One can also note that in contrast to \cite{KKGU2017}, here we are able to determine not only the tangential component of $A$ on $\partial \Omega$, but the entire trace of $A$ on $\partial \Omega$. Our result is as follows. 

\begin{prop}
\label{boundaryresult}
Let $\Omega \subseteq \R^n, n \geq 3$, be a bounded open set with $C^\infty$ boundary, and let $A_1, A_2 \in (H^1 \cap L^\infty) (\Omega; \C^n)$. Suppose that the assumption (A) holds for both operators $L_{A_1, q_1}$ and $L_{A_2, q_2}$, and that $\Lambda_{A_1, q_1} = \Lambda_{A_2, q_2}$. Then $A_1|_{\partial \Omega} = A_2|_{\partial \Omega}$ in $H^{1/2}(\partial \Omega; \C^n).$
\end{prop}
\begin{proof}
Let $f\in H^{1/2}(\p \Omega)$.  First, arguing similarly to  \cite[Appendix A]{KKGU2017geo} and using that $\Lambda_{A_1, q_1} = \Lambda_{A_2, q_2}$, we obtain the following integral identity, 
\begin{equation}
\label{A6}
\int_\Omega [-2i(A_1 \cdot \nabla u_1)\overline{v} + q_1u_1\overline{v}]dx = \int_\Omega [-2i(A_2 \cdot \nabla u_2)\overline{v} + q_2u_2\overline{v}]dx, 
\end{equation}
valid for all $u_1, u_2 \in H^1(\Omega)$ solving 
\begin{equation}
\label{A1}
\begin{cases} 
(-\Delta -2iA_j \cdot \nabla+q_j)u_j = 0,\\
u_j|_{\p\Omega}=f
\end{cases}
\end{equation}
for $j = 1, 2$, and all $v \in H^1(\Omega)$ solving 
\begin{equation}
\label{A2}
\begin{cases} 
-\Delta v = 0,\\
v|_{\p\Omega}=f.
\end{cases}
\end{equation}

Similarly to \cite{KKGU2017geo}, we shall construct some special solutions to \eqref{A1} and \eqref{A2}, whose boundary values have an oscillatory behavior while becoming increasingly concentrated near a fixed boundary point $x_0 \in \p\Omega$, see also \cite{Brown2001} and \cite{Brownsalo}. To this end, it is convenient to straighten out the boundary locally by means of the boundary normal coordinates. 

Let $y = (y', y_n) \in \R^n, y' = (y_1, ..., y_{n-1})$ be the boundary normal coordinates centered at $x_0$. Thus, $y$ varies in a neighborhood of 0 in $\R^n$. In terms of $y$, locally near $x_0$, the boundary $\p\Omega$ is defined by $y_n = 0$, and $y_n >0$ if and only if $x \in \Omega$. In what follows, we shall write $x = (x', x_n)$ instead of $y = (y', y_n)$.

Let $\eta \in C_0^\infty (\R^n; \R)$ be a function such that supp($\eta)$ is in a small neighborhood of 0, and 
\begin{align*}
\int_{\R^{n-1}} \eta (x', 0)^2 dx' = 1.
\end{align*}

As in \cite{Brown2001}, \cite{Brownsalo}, \cite{KKGU2017geo}, \cite{KKGU2017}, we let 
\begin{align*}
v_0(x) = \eta\bigg(\frac{x}{\lambda^{1/2}}\bigg)e^{\frac{i}{\lambda}(\tau' \cdot x'+ix_n)}, 0 < \lambda \ll 1,
\end{align*}
where $\tau ' \in \R^{n-1}$ = $T_{x_0} \partial \Omega$. This implies $v_0 \in C^\infty_0(\R^n)$ and supp($v_0$) is in $\mathcal{O}(\lambda^{1/2})$ neighborhood of $x_0 =0$. 

Let $f = v_0|_{\p\Omega}$. Then $v = v_0 + v_1$ solves \eqref{A2} if $v_1 \in H_0^1(\Omega)$ is the unique solution to the Dirichlet problem 
\begin{equation}
\label{A7}
\begin{cases} 
-\Delta v_1 = \Delta v_0 \quad \text{in}\quad\Omega,\\
v_1|_{\p\Omega}=0.
\end{cases}
\end{equation}
We shall need the following estimates established in \cite{Brownsalo}, \cite{KKGU2017geo}, and \cite{KKGU2017}, 
\begin{equation}
\label{A10}
\|v_0\|_{L^2(\Omega)} \leq \mathcal{O}(\lambda^{\frac{n-1}{4}+\frac{1}{2}}),
\end{equation}
\begin{equation}
\label{A11}
\|v_1\|_{L^2(\Omega)} \leq \mathcal{O}(\lambda^{\frac{n-1}{4}+\frac{1}{2}}). 
\end{equation}

Turning the attention to the problem \eqref{A1},  we see that  $u_j = v_0+w_j$ solves \eqref{A1} if $w_j \in H^1_0(\Omega)$ is the unique solution to \begin{equation}
\label{A8}
\begin{cases} 
(-\Delta -2iA_j \cdot \nabla+q_j)w_j = -(-\Delta -2iA_j \cdot \nabla+q_j)v_0 \quad \text{in}\quad \Omega,\\
w_j|_{\p\Omega}=0.
\end{cases}
\end{equation} 
Using the Lax-Milgram lemma together with the uniqueness of solution to \eqref{A8}, we obtain that 
\begin{equation}
\label{A9}
\|w_j\|_{H^1(\Omega)} \leq C\|(-\Delta -2iA_j \cdot \nabla+q_j)v_0\|_{H^{-1}(\Omega)}. 
\end{equation} 
To bound the right hand side of \eqref{A9}, let us recall that the following estimate was established in \cite[Appendix]{KKGU2017geo},
\begin{equation}
\label{A14}
\|(-\Delta -2iA_j \cdot \nabla)v_0\|_{H^{-1}(\Omega)}\le \mathcal{O}(\lambda^{\frac{n-1}{4}}). 
\end{equation}
Using \eqref{A10}, we see that 
\begin{equation}
\label{A17}
\|q_j v_0\|_{H^{-1}(\Omega)} \leq \mathcal{O}(\lambda^{\frac{n-1}{4}+\frac{1}{2}}).
\end{equation}
It follows from \eqref{A9}, \eqref{A14}, and \eqref{A17} that 
\begin{equation}
\label{A18}
\|w_j\|_{H^1(\Omega)} \leq \mathcal{O}(\lambda^{\frac{n-1}{4}}). 
\end{equation}

Now let us plug the solutions $u_j=v_0+w_j$ and $v=v_0+v_1$ of \eqref{A1} and \eqref{A2}, respectively, into  \eqref{A6}, multiply it by $\lambda^{-\frac{(n-1)}{2}}$, and compute the limit as $\lambda \rightarrow 0$. To this end, using \eqref{A10}, \eqref{A11}, and \eqref{A18}, we first observe that 
\begin{equation}
\label{eq_A19}
\begin{aligned}
	&\bigg|\lambda^{-\frac{(n-1)}{2}}\int_\Omega q_ju_j\overline{v}dx\bigg| = \bigg|\lambda^{-\frac{(n-1)}{2}}\int_\Omega q_j(v_0+w_j)(\overline{v_0} + \overline{v_1})dx\bigg|\\ &\leq \lambda^{-\frac{(n-1)}{2}}\|q_j\|_{L^\infty(\Omega)}\big(\|v_0\|_{L^2(\Omega)}+\|w_j\|_{L^2(\Omega)}\big)\big(\|v_0\|_{L^2(\Omega)}+\|v_1\|_{L^2(\Omega)}\big)\\
	&\leq \mathcal{O}(\lambda^{1/2}) \rightarrow 0 \quad \text{as} \quad \lambda \rightarrow 0. 
\end{aligned}
\end{equation}
Using \eqref{eq_A19}, we conclude from \eqref{A6} that 
\[
\lim_{\lambda\to 0}\lambda^{-\frac{(n-1)}{2}}\int_{\Omega} (A_1\cdot \nabla u_1)\overline{v}dx=\lim_{\lambda\to 0}\lambda^{-\frac{(n-1)}{2}}\int_{\Omega} (A_2\cdot \nabla u_2)\overline{v}dx,
\]
which is exactly the same as \cite[formula (A.22)]{KKGU2017geo}. The arguments in \cite{KKGU2017geo} allow us therefore to conclude from \eqref{A6} that
\begin{align*}
	(\tau', i) \cdot A_1(0) = (\tau', i) \cdot A_2(0),
\end{align*}
for all $\tau' \in \R^{n-1}$. This completes the proof of Proposition \ref{boundaryresult}.
\end{proof}

\section*{Acknowledgements}  
The author would like to thank Katya Krupchyk for her support and guidance.  The research is partially supported by the National Science Foundation (DMS 1500703, DMS 1815922). The author is also very grateful to the referee for the helpful comments which led to improvements in the presentation of the paper.

\end{document}